\documentclass[12pt,intlimits,reqno]{article}
\usepackage{amsmath}
\usepackage{amssymb,amsthm}
\usepackage[T1]{fontenc}
\usepackage[utf8]{inputenc}
\usepackage{microtype}
\usepackage{enumerate}
\usepackage{verbatim}
\usepackage{bbm}
\usepackage{antpolt}

\usepackage[dvipsnames]{xcolor}
\usepackage[colorlinks=true, urlcolor=blue,citecolor=OliveGreen,menucolor=]{hyperref}

\newcommand{\pair}[2]{\left\langle #1\;;\;#2 \right\rangle}

\newcommand{\R}{\mathbb R}

\newcommand{\N}{\mathbb N}

\newcommand{\Hh}{\mathcal H}

\newcommand{\pb}{\{\,\cdot\,,\,\cdot\,\}}
\newcommand{\cmm}{[\,\cdot\,,\,\cdot\,]}
\renewcommand{\to}{\rightarrow}

\renewcommand{\epsilon}{\varepsilon}
\renewcommand{\phi}{\varphi}

\DeclareMathOperator{\id}{id}

\DeclareMathOperator{\ad}{ad}

\newcommand{\be}{\begin{equation}}
\newcommand{\ee}{\end{equation}}
\newcommand{\bse}{\begin{subequations}}
\newcommand{\ese}{\end{subequations}}
\newcommand{\ben}{\begin{enumerate}}
\newcommand{\een}{\end{enumerate}}
\newcommand{\bit}{\begin{itemize}}
\newcommand{\eit}{\end{itemize}}

\newtheorem{thm}{Theorem}[section]

\newtheorem{lem}[thm]{Lemma}
\newtheorem{prop}[thm]{Proposition}
\theoremstyle{definition}

\newtheorem{defn}[thm]{Definition}
\theoremstyle{remark}
\newtheorem{rem}[thm]{Remark}

\numberwithin{equation}{section}

\newcommand{\E}{\mathbb E}
\newcommand{\M}{\mathbb M}

\begin{document}
  \title{Poisson structure on predual of Banach Lie algebroid}
  \author{Tomasz Goli\'nski, Grzegorz Jakimowicz}
   \date{\small University of Bia{\l}ystok, Faculty of Mathematics\\
    Ciołkowskiego 1M, 15-245 Bia{\l}ystok, Poland\\
    email: tomaszg@math.uwb.edu.pl, g.jakimowicz@uwb.edu.pl
  }
\maketitle

\begin{abstract}
We construct the linear Poisson bracket on the predual bundle of a Banach Lie algebroid. It is an alternative approach to the already known results on the linear sub-Poisson structure on the dual bundle. We also discuss the existence of queer Banach Lie algebroids. The example on a trivial bundle is presented and the situation on precotangent bundles is discussed.
\end{abstract}

\tableofcontents
  
\section{Introduction}

There is a deep connection between the theory of Poisson manifolds and the theory of Lie groupoids and algebroids. The underlying observation is that an important class of linear Poisson brackets (called Lie--Poisson spaces) is obtained starting from Lie algebras. An extension of this construction is the linear Poisson structure defined by a Lie algebroid. It was investigated in multiple papers, including \cite{karasev,coste87groupoides,courant90,grabowski97,marle,deleon2010}.

Unlike the finite-dimensional case, in the Banach case in general the dual of a Banach Lie algebra no longer carries the canonical linear Poisson structure. One needs to consider a predual space instead. Moreover, an additional condition is required to ensure the existence of Hamiltonian vector fields, see \cite[Theorem 4.2]{OR}. The notion of Banach Lie--Poisson space is useful e.g. in the study of integrable systems, see \cite{beltita05,Ratiu-grass,Oind,GO-grass,DO-L2,tumpach-bruhat,GT-momentum,GT-partiso-bial}.

In the Banach context, there are even more pitfalls and subtleties. Even the notion of a Banach Poisson manifold has several competing definitions. In this paper, we will use the original one, introduced in \cite{OR} and further clarified in \cite{BGT} (see Definitions~\ref{def:pb} and \ref{def:pm}). In the literature, there are other approaches as well, see \cite{GRT-poisson-bial} for a recent comparison. Notably, many of those approaches define the Poisson bracket only on some subalgebra of functions on the given manifold. In this paper, we will refer to such brackets as sub-Poisson brackets, following \cite{pelletier} (see Definition~\ref{def_spb}). Brackets of this kind emerge naturally in various aspects, e.g. while studying weak symplectic structures or in the context of Banach Poisson--Lie groups \cite{tumpach-bruhat,GT-AKS}.

This paper addresses the relationship between the notions of Banach Lie algebroid and Banach Poisson manifold. One approach to this construction was presented in \cite{pelletier,cabau-pelletier_book}, where a sub-Poisson structure on the bundle dual to the Banach Lie algebroid was defined. In this paper, we present an alternative approach by considering a predual bundle. This leads to considerable simplification of the resulting structure: we obtain a Poisson bracket defined for all smooth functions in the sense of Definition~\ref{def:pb}. It is consistent with the situation presented in \cite{OR} for Banach Lie algebras. As an example, we describe the Poisson bracket obtained on the predual to the tangent bundle with respect to canonical symplectic and algebroid structures.

In Section~\ref{sec:basic} we begin by stating the necessary basic definitions of objects under consideration: Banach Poisson brackets and Banach Lie algebroids. We also discuss some of the infinite-dimensional peculiarities of Poisson geometry, such as the existence of queer Poisson brackets depending on higher order derivatives (see \cite{BGT}). We address the possibility of existence of a queer Banach Lie algebroid, giving a partial negative result.

In Section~\ref{sec:poisson} the construction of the linear Poisson structure on the predual bundle to a Banach Lie algebroid is described. Necessary and sufficient conditions for this bracket to define the structure of a Banach Poisson manifold are given.

In Section~\ref{sec:alg} the converse path is discussed. Namely, given a Banach Poisson bundle with linear Poisson bracket, the Banach Lie algebroid structure on the dual bundle is constructed.

We conclude the paper with two examples. First, in Section~\ref{sec:weak_symplectic}, we discuss the case of canonical weak symplectic structure on a precontangent bundle. We give the formula for the Poisson bracket given by the Lie algebroid structure of the tangent bundle and comment on its relationship with the sub-Poisson bracket defined by the canonical weak symplectic form. In Section~\ref{sec:trivial} we discuss a certain class of Lie algebroid structures on the trivial Banach bundle $\ell^2\times\ell^\infty\to\ell^2$ and determine when they lead to the structure of a Banach Poisson manifold.

\section{Basic notions}\label{sec:basic}

Throughout the paper, we will use the following notation. Let $M$ be a smooth Banach manifold modeled on the Banach space $\M$. We will denote with $E$ a Banach bundle $\pi:E\to M$ with the typical fiber $\E$. There exists the associated dual Banach bundle $\pi^*:E^*\to M$ with the typical fiber $\E^*$. We will postulate the existence of a predual space $\E_*$ to the Banach space $\E$ and a predual bundle $\pi_*:E_*\to M$ to $E$. We will be interested in a non-reflexive case, i.e. the situation when $\E_*$ is not canonically isomorphic to $\E^*$. Note that $E_*$ might be non-unique.

\subsection{Banach Poisson manifolds}

There are several difficulties in studying Poisson structures on Banach manifolds. One of them is the lack (in general) of bump functions (see e.g. \cite{kurzweil1954,bonic-frampton1966} or discussion in \cite{pelletier,BGJP}). In consequence, locally defined functions may not possess any extension to the whole manifold. As far as we know, it is an open problem even if there is enough of smooth globally defined functions for their differentials to span the cotangent bundle. One approach to this problem is to assume explicitly that the Poisson bracket is localizable, i.e. it preserves the sheaf of locally-defined functions.

Moreover, the existence of the Poisson tensor does not follow automatically from the usual properties of the Poisson bracket (i.e. Leibniz identity) and needs to be assumed separately. The counterexamples, known as \emph{queer Poisson brackets}, were discussed in \cite{BGT}. In that case, the value of Poisson bracket $\{f,g\}$ depends on higher derivatives of $f$ and $g$.

In consequence, there are several approaches to the notion of a Poisson bracket on a Banach manifold. In this paper, following \cite{OR} and \cite{BGT}, we will use the following definitions:
\begin{defn}\label{def:pb}
A (localizable, non-queer) Poisson bracket on a smooth Banach manifold $P$ is an antisymmetric bilinear mapping
$\pb:C^\infty(P)\times C^\infty(P)\to C^\infty(P)$ given by a Poisson tensor $\Pi\in\Gamma(\bigwedge^2T^{**}P)$ 
\be\label{pb_tensor}\{f,g\} = \Pi(df,dg)\ee
such that $\pb$ satisfies Jacobi identity.
\end{defn}

\begin{defn}\label{def:pm}
A Banach Poisson manifold is a smooth Banach manifold $P$ equipped with a Poisson bracket $\pb:C^\infty(P)\times C^\infty(P)\to C^\infty(P)$ such that the map 
$\sharp:T^*P\to T^{**}P$ given by
\be \label{sharp_def}\sharp\mu :=\Pi(\cdot, \mu)\ee
takes values in predual space
\be \label{sharp} \sharp(T^*P)\subset TP.\ee
\end{defn}

The condition \eqref{sharp} guarantees the existence of Hamiltonian vector fields for the Poisson structure.

There are also many other, more general, notions of Poisson brackets on Banach manifolds, see \cite{GRT-poisson-bial} for a recent comparison of those approaches. Some of those brackets are defined only on a particular class of functions. Those brackets are particularly useful to the study of Banach Poisson--Lie groups, see \cite{tumpach-bruhat,GT-u-bial,GT-AKS}. In order to make it easier to distinguish them, in the present paper they will be referred to as sub-Poisson brackets (following the terminology of \cite{pelletier}):

\begin{defn}\label{sub_Poisson}
Let $T^\flat M$ be a Banach subbundle of the cotangent bundle $T^*M$ and let $P$ be a bundle morphism $P: T^\flat M \rightarrow TM$ that is skew-symmetric relatively to the duality pairing, i.e. $\langle\alpha,P\beta\rangle = -\langle\beta, P\alpha\rangle$ for any $\alpha,\beta\in T^\flat M$. Consider the following subalgebra 
\[\mathcal{F}^\flat = \{ f\in C^\infty(M), df \in \Gamma(T^\flat M)\}\subset C^\infty(M).\]

We say that $P$ defines a sub-Poisson bracket
\begin{equation}\label{def_spb}
\{ f, g\} = \langle df, P dg\rangle, 
\end{equation}
where $f, g\in \mathcal{F}^\flat$, if $\pb$ satisfies the Jacobi identity.
\end{defn}

The approaches to Poisson structures for other classes of infinite-dimensional manifolds can be found e.g. in \cite{neeb14,pelletier19,cabau-pelletier_book}.

In this paper, we will be interested in Poisson brackets on a Banach vector bundle $E_*$ predual to a Banach vector bundle $E$ over a Banach manifold $M$. In this situation, one may distinguish a family of locally defined smooth fiber-wise linear functions $\mathcal L_{loc}(E_*)$ in 
$C^\infty_{loc}(E_*)$. 
\begin{defn}
If a Poisson bracket $\pb$ on $E_*$ preserves the family of locally defined smooth fiber-wise linear functions
\[ \{\mathcal L_{loc}(E_*), \mathcal L_{loc}(E_*)\} \subset \mathcal L_{loc}(E_*), \]
we say that it is a linear Poisson bracket.
\end{defn}

In the same spirit, a linear sub-Poisson bracket would mean a sub-Poisson bracket that preserves the family $\mathcal L_{loc}(E_*) \cap \mathcal F^\flat$.

\subsection{Banach Lie algebroids}

Consider a Banach vector bundle $E$ over a Banach manifold $M$. We will denote by $j^1_m s$ the first jet of a local section $s\in\Gamma_{loc}(E)$ at the point $m\in M$, i.e. an equivalence class of sections, which coincide with $s$ up to terms of order 1:
\[ j^1_m s = \{ t\in \Gamma_{loc}(E) \;|\; t(m)=s(m), T_mt = T_ms \}. \]
The space of all first jets of functions at an arbitrary point will be denoted $\mathcal J^1(E)$ and forms a Banach bundle over $M$, see \cite[Section 1.12.1]{cabau-pelletier_book}.

Lie algebroids are infinitesimal counterparts of Lie groupoids and the notion comes from \cite{pradines67}. In Banach context it was studied e.g. in \cite{Anastasiei,pelletier,BGJP,OJS,OJS-fiber}. In the convenient setting they appeared e.g. in \cite[Definition 3.76]{cabau-pelletier_book} and recently in Bastiani setting in \cite{shaltut2026}. Note that, just as for Poisson brackets, there may be a possibility of \emph{queer Lie algebroids} with bracket depending on higher derivatives, which we exclude from our considerations.
We also assume that the Lie algebroid bracket is localizable in the sense that it preserves the sheaf of sections of $\pi^{-1}(U)$ for the open sets $U\subset M$, see \cite{pelletier,BGJP} for details.

Both localizability and non-queerness conditions are automatically satisfied for Banach Lie algebroids of Banach Lie groupoids, see \cite{BGJP} and \cite[Remark 3.11]{cabau-pelletier_book}.

\begin{defn}\label{def:alg}
Let $E$ be a Banach vector bundle over $M$. The structure of a Banach Lie algebroid on $E$ is given as a localizable Lie bracket of sections $\cmm: \Gamma(E)\times\Gamma(E)\to \Gamma(E)$ and a smooth bundle morphism covering identity (called \emph{anchor}) $a: E \to TM$ such that
\ben[1.]
  \item $(\Gamma(E),\cmm)$ is a Banach Lie algebra,
  \item $[X,fY]=a(X)f\cdot Y + f[X,Y]$ for all $X,Y\in \Gamma(E)$, $f\in C^\infty(M)$,
  \item $\cmm$ depends only on the first jet of sections.
  \een
\end{defn}

A structure satisfying conditions 1. and 2., but not condition 3. would be called a \emph{queer Banach Lie algebroid}. Following the paper \cite{BGT}, where the construction of queer Poisson brackets depending on higher derivatives was exhibited, one could also expect the possibility of existence of Banach Lie algebroids with the analogous behaviour, see also \cite[Remark 3.11]{cabau-pelletier_book}. However, it turns out that for a wide class of Banach bundles no such situation can occur. The proof of this fact is analogous to \cite[Proposition 2.2.2]{marle}, but requires an assumption of existence of a Schauder basis. This assumption holds for most typical separable Banach spaces, but fails for all non-separable Banach spaces (see e.g. \cite{lindenstrauss}).

\begin{thm}
If the typical fiber $\E$ of the Banach vector bundle $E\to M$ admits a Schauder basis, then there are no queer Banach Lie algebroids on the bundle $E$.
\end{thm}
\begin{proof}
Assume that $\cmm$ is a Lie algebroid bracket satisfying conditions 1. and 2. of Definition~\ref{def:alg}, but not condition 3. Explicitly it means that for some $m\in M$ there exist local sections $s,t\in \Gamma(E)$ such that $[s,t](m)\neq 0$ while $s\in j^1_m 0$.

Let us consider a trivialization of $E$ around a point $p\in E$. In this trivialization, the section $s$ can be viewed as a map $\tilde s: \M\to\E$. Denote by $\{e_n\}_{n\in\N}$ a Schauder basis of $\E$. Decomposing the section $s$ in that basis, we get
\[ \tilde s = \sum_{n=1}^\infty s_n e_n, \]
where $s_n:\M \to \R$ are smooth since they can be written as $e_n^*(\tilde s)$ for biorthogonal functionals $\{e_n^*\}_{n\in\N}$ related to the basis $\{e_n\}_{n\in\N}$, see e.g. \cite[Section 1.b]{lindenstrauss}. Moreover, from the assumed properties of $s$ it follows that $s_n(m) = 0$ and $ds_n(m) = 0$.

Using the properties of the Lie bracket (condition $2.$ from Definition~\ref{def:alg}), one gets
\[
[\tilde s, \tilde t] = \sum_{n=1}^\infty [s_n e_n, \tilde t] = - a(t)s_n \cdot e_n + s_n [e_n,\tilde t].
\]
Since $a(t)s_n= ds_n(a(t))$, we observe that evaluating this expression at the point $m$, both terms vanish and we get a contradiction.
\end{proof}

The existence of queer Banach Lie algebroids on bundles with fibers not admitting a Schauder basis is still an open problem. Note that the existence of a Hamel basis is not sufficient as in general coefficients $s_n$ would not be smooth.

\subsection{Correspondence between Poisson and Lie algebroid structures}

In the finite-dimensional setting there is an equivalence between linear Poisson structures and Lie algebroid structures on the dual bundle, see e.g. \cite{coste87groupoides,courant90,marle,mackenzie2005,deleon2010}. 

In the infinite-dimensional setting the situation is not that simple. Given a linear sub-Poisson structure on a certain subbundle $T^\flat E^*$ of the cotangent bundle $T^*E$ of a Banach vector bundle $E$, it was shown in \cite[Theorem 4.8]{pelletier} that there exists a canonically defined Banach Lie algebroid structure on $E$. Moreover, given a Banach Lie algebroid structure on $E$ one obtains a sub-Poisson structure on $T^\flat E^*$. That result is an adaptation of the construction known from the finite-dimensional case. The proof in Banach context in \cite[Theorem 4.8]{pelletier} silently assumes that Lie bracket and Poisson bracket are non-queer, i.e. they depend only on the first jet, which is later clarified (and generalized to the convenient setting) in \cite[Theorem 7.1]{cabau-pelletier_book}. 

For this construction two families of functions are used. In the finite-dimensional case their differentials always span the whole fibers of the cotangent bundle. However, in the Banach context, they only span a certain subspace denoted as $T^\flat E^*$. In effect, the Poisson bracket in that setting was defined only for a certain class of smooth functions on $E^*$. It is a similar situation to sub-Riemannian geometry, where a metric is defined only for a certain class of tangent vectors. In this paper, we show that working with the predual bundle $E_*$ instead of $E^*$, it is possible to obtain a Poisson structure in the sense of \cite{OR}. More precisely, we demonstrate that assuming that there exists a predual vector bundle $E_*$, one has
\[
T^\flat E^*_{|E_*} = T^* E_*.
\]

\subsection{Local picture}
Let us introduce local coordinates in the considered Banach vector bundles: take any open set $U\subset M$ such that $E_U:=\pi^{-1}(U)$ is trivial, i.e.
\be E_U\cong U\times \E.\ee
Let us write down some natural explicit identifications for other related bundles, which will be useful in next sections. The dual and predual bundles can be locally identified with the following sets:
  \[ E^*_U\cong U\times \E^*\]
  \[ {E_*}_U\cong U\times \E_*\]
Similar identifications are also valid for tangent and cotangent bundles of $E$ and $E_*$:
  \[ TE_U \cong U\times \E \times \M \times \E \]
  \[ TE^*_U \cong U\times \E^* \times \M \times \E^*  \]
  \[ T{E_*}_U \cong U\times \E_* \times \M \times \E_*  \]
  
  \[ T^*E_U \cong U\times \E \times \M^* \times \E^* \]
  \[ T^*E^*_U \cong U\times \E^* \times \M^* \times \E^{**}  \]
  \[ T^*{E_*}_U \cong U\times \E_* \times \M^* \times \E  \]
Finally the jet bundle of $E$ can be identified with:
\[ \mathcal J^1(E_U) \cong U \times \E \times L(\M,\E),\]
where $L(\cdot,\cdot)$ stands for the Banach space of bounded linear maps.

In the paper, to simplify the notation, we will not write explicitly the trivialization charts in the formulas.

Let us state here a Banach version of a fact observed in the convenient setting in \cite[Notation 3.10]{cabau-pelletier_book}.

\begin{prop}\label{prop:lie_bracket_local}
Consider two sections $X,Y\in\Gamma(E_U)$, which we express in the trivialization $U\times \E$ as $X(m)=(m,v_m)$ and $Y(m)=(m,w_m)$. The local expression of the Lie bracket of a Banach Lie algebroid $E\to M$ is as follows:
\be \label{lie_bracket_local}[X, Y](m) = \big(a(v_m) Y\big)(m) - \big(a(w_m) X\big)(m) + C_m(v_m,w_m) \ee
for $m\in U\subset M$, where $U\ni m\mapsto C_m\in L_{skew}(\E, \E; \E)$ is a certain smooth field of bounded skew-symmetric bilinear maps.
\end{prop}

\begin{rem}\label{rem:not_lie_algebra}
Note that the field of bilinear maps $m\mapsto C_m$ depends on the choice of trivialization. In general $C_m$ might not define a Lie bracket on $\E$ because it might not satisfy the Jacobi identity. However, if the Lie algebroid bracket preserves the family of sections constant in a chosen trivialization, $C_m$ would define a structure of Lie algebra on $\E$. This situation can happen e.g. if $E$ is a sheaf of Lie algebras, i.e. $a=0$, see \cite{marle}.
\end{rem}

\section{Poisson structure on a predual to Lie algebroid}\label{sec:poisson}

There exist two distinguished families of local functions on the dual bundle $E^*$: one given as pull-backs of local functions on the base 
\[ f\circ \pi^* \] 
for $f\in C_{loc}^\infty(M)$ and the other defined by pairing with local sections of $E$: 
\[ \lambda_X(\rho)=\pair\rho{X_{\pi^*(\rho)}}\]
for $X\in \Gamma_{loc}(E)$. They are respectively fiber-wise constant and linear. Those families are useful for describing the relationship between the Lie algebroid structure on $E$ and the Poisson structure on $E^*$.

Using those functions, given a Banach Lie algebroid structure on $E$, a sub-Poisson structure on $E^*$ was constructed in \cite{pelletier} (with some gaps in the construction clarified later in \cite{cabau-pelletier_book}). It is an adaptation of the construction known from the finite-dimensional case, see e.g. \cite{coste87groupoides,courant90,grabowski97,marle,mackenzie2005,deleon2010}.

In general there is an obstacle in the Banach case: the differentials of functions $f\circ \pi^*$ and $\lambda_X$ at some point $\rho$ do not span the whole dual fiber $T^*_\rho E^*$, see Remark~\ref{rem:not_span} below.

Thus the sub-Poisson bracket in \cite[Theorem 4.8]{pelletier} was defined only for a certain class of smooth functions on $E^*$, i.e. one considers a subbundle $T^\flat E^*$ spanned by differentials of $f\circ\pi^*$ and $\lambda_X$. Namely, the relevant part of that theorem, in our notation, states:
\begin{thm}[Theorem 4.8, \cite{pelletier}]
Each Banach Lie algebroid $E$ defines a unique linear sub-Poisson bracket on the subalgebra $\mathcal F^\flat$ given by $T^\flat E^*$.
\end{thm}

In this paper, we explore the counterpart of this theorem for $E_*$.

\begin{thm}\label{functions-span}
The differentials of functions $f\circ \pi_*$ and $\lambda_X$, for $f\in C_{loc}^\infty(M)$, $X\in \Gamma_{loc}(E)$ span the cotangent bundle $T^* E_*$ of the predual bundle $E_*$, i.e. for any $\rho\in E_*$ we have:
\be T^*_\rho E_* = \operatorname{span} \big(\{d(f\circ \pi_*)(\rho)\;|\;f\in C_{loc}^\infty(M)\}\cup \{
    d\lambda_X(\rho)\;|\; X\in \Gamma(E) \}\big).\ee
\end{thm}
\begin{proof}
Since we consider functions defined locally, it is enough to compute the differentials of $f\circ \pi_*$ and $\lambda_X$ in a trivialization. Recall that for brevity, we omit trivialization charts in the formulas. Consider $\rho=(m,\phi)\in U\times \E_*$ and $X(m)=(m,v_m)\in U\times \E$. By direct calculation, we obtain 
\be \label{df}d(f\circ \pi_*)(\rho)=(m,\phi, df(m), 0),\ee
\be \label{dlambda}d\lambda_X(\rho)=(m,\phi, d(\phi\circ v)(m) ,v_m),\ee
where we consider $\phi\circ v$ as a map $U\to \R$. In this manner its differential at $m$ is an element of $\M^*$.
The differentials of the functions of the first family span $\{m\}\times\{\phi\}\times \M^*\times \{0\}$ (see e.g. \cite[Corollary 4.2.14]{ratiu-mta3}), while the differentials of the functions of the second family span $\{m\}\times\{\phi\}\times D \times \E$ for some $D\subset \M^*$. Thus both families of functions together span $\{m\}\times\{\phi\}\times \M^* \times \E\cong T^*_\rho E_*$.
\end{proof}

\begin{rem}\label{rem:not_span}
Note that from the preceding proof it also follows that, in general, the differentials of functions $f\circ \pi^*$ and $\lambda_X$ do not span the cotangent bundle $T^* E^*$ of the dual bundle $E^*$. Differentials of both families of functions together at $\rho=(m,\phi)$ span only $\{m\}\times\{\phi\}\times \M^* \times \E \varsubsetneq \{m\}\times\{\phi\}\times \M^* \times \E^{**} \cong T^*_\rho E^*$. See also \cite[Proposition 4.7]{pelletier}.
\end{rem}

\begin{thm}\label{algebroid-poisson}
There exists a canonical linear Poisson bracket on $E_*$ related to the Banach Lie algebroid $E$ satisfying

\[\{f\circ \pi_*,g\circ \pi_*\}=0,\]
\be\{\lambda_X,f\circ \pi_*\}=(a(X)f)\circ \pi_*,\label{alg-bracket}\ee
\[\{\lambda_X,\lambda_Y\}=\lambda_{[X,Y]}.\]
\end{thm}
\begin{proof}
The Poisson bracket of functions can be defined via the Poisson tensor $\Pi\in \Gamma( \bigwedge\limits^2 T^{**}E_*)$ by \eqref{pb_tensor}. Following \cite{pelletier}, we define the tensor $\Pi$ on differentials of $f\circ \pi_*$ and $\lambda_X$ as:
\[\Pi(d(f\circ \pi_*),d(g\circ \pi_*))=0,\]
\[\Pi(d\lambda_X,d(f\circ \pi_*))=(a(X)f)\circ \pi_*,\]
\[\Pi(d\lambda_X,d\lambda_Y)=\lambda_{[X,Y]}.\]
Observe that, since we excluded queer algebroids in Definition~\ref{def:alg}, the value of $\lambda_{[X,Y]}$ depends only on the differentials $d\lambda_X$, $d\lambda_Y$. Similarly, one has $(a(X)f)\circ \pi_* = \pair{df\circ \pi_*}{a(X)}$.
Thus, by Theorem~\ref{functions-span}, $\Pi$ is a well-defined bivector field on $T^*_\rho E_*$. The only thing left to check is the Jacobi identity. The proof of this fact is analogous to the one in \cite{pelletier}. Namely, the relations \eqref{alg-bracket} imply that the Poisson bracket satisfies the Jacobi identity for functions of the form $\lambda_X + f\circ\pi_*$. Since the differentials of those functions span $T^*E_*$ and the Poisson bracket depends only on the differentials, we conclude that the Jacobi identity holds for all functions.

This bracket satisfies the conditions of Definition~\ref{def:pb}, i.e. it is a localizable non-queer Poisson bracket. Moreover directly from the third condition of \eqref{alg-bracket}, we see that it is indeed linear.
\end{proof}
  
In order to have a structure of the Banach Poisson manifold according to Definition~\ref{def:pm}, we need to check the condition \eqref{sharp}. This condition guarantees the existence of Hamiltonian vector fields for any function $H\in C^\infty(E_*)$. In the case of the Poisson bracket defined by \eqref{alg-bracket}, we can explicitly compute the map $\sharp$ in a trivialization $E_U\cong U\times\E$.

\begin{lem}\label{lem:sharp}
The sharp map $\sharp: T^*E_*\to T^{**}E_*$ for the Poisson bracket defined by \eqref{alg-bracket} has the following expression in a trivialization:
\[\sharp(m,\phi,\mu,x) = \big(m,\phi,-a(x),a^*(\mu)-\ad_{m,x}^*\phi\big)\]
for $(m,\phi,\mu,x)\in U\times \E_*\times \M^*\times \E$, where by $a^*$ we denote the dual map $a^*: T^*M\to E^*$ to the anchor and by $\ad_{m,x}^*$ we denote the dual map $\ad_{m,x}^*:\E^*\to\E^*$ to the map $\ad_{m,x}:\E\to\E$ given by $\ad_{m,x}(y):=C_m(x,y)$ for $x,y\in \E$ (see Proposition~\ref{prop:lie_bracket_local}). 
\end{lem}
\begin{proof}
By direct computation with \eqref{sharp_def} and formulas obtained in the proof of Theorem~\ref{functions-span}, the formulas \eqref{alg-bracket} are equivalent to the following three equalities:
\[\sharp(m,\phi,df(m),0)(m,\phi,dg(m),0) = \{g\circ\pi_*,f\circ\pi_*\}(m,\phi)=0,\]
\[\sharp(m,\phi,df(m),0)(m,\phi,d(\phi\circ v)(m),v_m) = \{\lambda_v,f\circ\pi_*\}(m,\phi)=(a(v)f)(m),\]
\[\sharp(m,\phi,d(\phi\circ v)(m),v_m)(m,\phi,d(\phi\circ w)(m),w_m) = \{\lambda_w,\lambda_v\}(m,\phi)=\lambda_{[w,v]}(m,\phi).\]
The first of them implies that
\[ \sharp(m,\phi,df(m),0) = (m,\phi,0,\alpha)\]
for some value of $\alpha$.
From the second equality, we get that
\[ \langle \alpha; v_m \rangle = (a(v)f)(m) = \langle df(m) ; a(v_m) \rangle. \]
In consequence one obtains the following equality:
\[\sharp(m,\phi,df(m),0) = \big(m,\phi,0,a^*(df(m))\big).\]
By linearity of $\sharp$ in the fibers we note that
\[\sharp(m,\phi,d(\phi\circ v)(m),v_m) = \sharp(m,\phi,d(\phi\circ v)(m),0) + \sharp(m,\phi,0,v_m) = \]
\[ = \big(m,\phi,0,a^*(d(\phi\circ v)(m))\big) + \sharp(m,\phi,0,v_m). \]
Thus the third equality can be reformulated as follows:
\[\sharp(m,\phi,d(\phi\circ v)(m),v_m)(m,\phi,d(\phi\circ w)(m),w_m) = \]
\[ = \langle a^*(d(\phi\circ v)(m)) ; w_m \rangle + \sharp(m,\phi,0,v_m)(m,\phi,d(\phi\circ w)(m),w_m) = \lambda_{[w,v]}(m,\phi).\]
Note that using the definition of $\lambda$ and local presentation of the algebroid bracket \eqref{lie_bracket_local} we get:
\[ \lambda_{[w,v]}(m,\phi) = \langle \phi; [w,v](m)\rangle = 
\langle \phi; \big(a(w_m) v\big)(m) - \big(a(v_m) w\big)(m) + C_m(w_m,v_m)\rangle = \]
\[= \langle a^*(d(\phi\circ v)(m)) ; w_m \rangle - \langle d(\phi\circ w)(m) ; a(v_m \rangle) - \langle \ad^*_{m,v_m}\phi  ; w_m \rangle. \]
Plugging it into the previous formula we get:
\[ \sharp(m,\phi,0,v_m)(m,\phi,d(\phi\circ w)(m),w_m) = \lambda_{[w,v]}(m,\phi) - \langle a^*(d(\phi\circ v)(m)) ; w_m \rangle =\] 
\[ = - \langle (d(\phi\circ w)(m)) ; a(v_m) \rangle - \langle \ad^*_{m,v_m}\phi; w_m \rangle, \]
which finally means that
\[ \sharp(m,\phi,0,v_m) = (m,\phi, - a(v_m), - \ad^*_{m,v_m}\phi).\]
Taking it all together we conclude:
\[\sharp(m,\phi,d(\phi\circ v)(m),v_m) = \big(m,\phi,-a(v_m),a^*(d(\phi\circ v)(m))-\ad_{m,v_m}^*\phi\big),\]
and the postulated formula follows.
\end{proof}

Note that in general, the map $\ad_{m,x}$ is not the adjoint representation for any Lie algebra, see Remark~\ref{rem:not_lie_algebra}. 

\begin{rem}
An arbitrary function $f\in C_{loc}^\infty(E_*)$
in a trivialization can be seen as a function on $U\times \E_*\subset \M\times \E_*$. 
Its differential $df\in\Gamma(T^*E_*)\cong \Gamma(U\times \E_*\times \M^* \times E)$ can be decomposed with respect to the direct sum as
\[ df(m,\phi) = \left(\frac{\partial f}{\partial m},  \frac{\partial f}{\partial \phi}\right)\in \M^* \times \E,\]
where by $\tfrac{\partial f}{\partial m}$ (resp. $\tfrac{\partial f}{\partial \phi}$) we mean a partial derivative in the direction of the Banach subspace $\M$ (resp. $\E$).

\end{rem}
\begin{prop}\label{prop:pb_formula}
For two functions $f,g\in C^\infty(E_*)$ the Poisson bracket in trivialization assumes the form
\[ \{f,g\}(m,\phi) = \]
\be \label{alg-bracket-gen} 
= \pair{a(\tfrac{\partial f}{\partial \phi})}{\tfrac{\partial g}{\partial m}}-\pair{a(\tfrac{\partial g}{\partial \phi})}{\tfrac{\partial f}{\partial m}} +\pair{C_m(\tfrac{\partial f}{\partial \phi},\tfrac{\partial g}{\partial \phi})}{\phi}\ee
\end{prop}
\begin{proof}
Using the formulas \eqref{pb_tensor}, \eqref{sharp_def} and Lemma~\ref{lem:sharp} we conclude
\[ \{f,g\}(m,\phi) = -\pair{a(\tfrac{\partial g}{\partial \phi})}{\tfrac{\partial f}{\partial m}} + \]
\[ + \pair{a^*(\tfrac{\partial g}{\partial m})-\ad^*_{m,\tfrac{\partial g}{\partial \phi}}\phi}{\tfrac{\partial f}{\partial \phi}}.
\]
Moving the dual maps to the other argument and substituting the definition of $\ad_{m,x}$, we obtain the claimed formula.
\end{proof}

\begin{thm}\label{thm:pb}
The predual bundle $E_*$ to the Lie algebroid $E$ equipped with the canonical Poisson bracket \eqref{alg-bracket-gen} is a Banach Poisson manifold in the sense of Definition~\ref{def:pm} if and only if the following conditions hold
\be a^*(T^*M)\subset E_*,\ee
\be\ad_{m,x}^*(\E_*)\subset \E_*,\ee
\end{thm}
\begin{proof}
It follows directly from Lemma~\ref{lem:sharp}. In general $\sharp(m,\phi,\mu,x)$ is an element of $U\times \E_*\times \M^{**}\times \E^*$. In order for condition \eqref{sharp} of Definition~\ref{def:pm} to hold, $\sharp(m,\phi,\mu,x)$ needs to belong to $U\times \E_*\times \M\times \E_*$. Thus one needs
\[ a(x) \in \M, \]
\[ a^*(\mu)-\ad_{m,x}^*\phi \in \E_*. \]
The first of those conditions is automatic by the definition of anchor map, the other is equivalent to the given conditions.
\end{proof}

There are two extreme examples of Lie algebroids. One is a Lie algebra with $M$ being a singleton set. The other is the tangent bundle with the anchor being identity map and the Lie bracket given as the commutator of vector fields. Let us briefly comment on both situations.

In the case where $E=\E$ is a Banach Lie algebra, the map $\ad_{m,x}$ is the usual adjoint representation $\ad_x = [x,\cdot]$ and this theorem reduces to Theorem 4.2 in \cite{OR}. Since the anchor map is zero in that case, only the second condition of Theorem~\ref{thm:pb} survives and assumes the form
\[\ad^*_x \E_*\subset \E_*\]
for $x\in \E$. The resulting Banach Poisson manifold is called a Banach Lie--Poisson space with bracket assuming the usual form
\[ \{f,g\}(\phi) = \pair{[df(\phi),dg(\phi)]}{\phi}. \]

Another typical example can be obtained when we take as $E$ a tangent bundle $TM$ of a Banach manifold $M$. In this case, there is a canonical Banach Lie algebroid structure defined by taking $\cmm$ as the commutator of vector fields and $a=\id$. Assuming that the precotangent bundle $T_*M$ exists, from Theorem~\ref{algebroid-poisson} we obtain a Poisson bracket on the bundle $T_*M$ related to the canonical weak symplectic structure on $T^*M$. We will explore this situation in detail in Section~\ref{sec:weak_symplectic}.

\section{Banach Lie algebroid structure on the dual of a linear Poisson manifold}\label{sec:alg}

We will address now the opposite construction. Starting from a linear localizable non-queer Poisson bracket on the bundle $E_*$, we will describe the construction of a Banach Lie algebroid structure on $E$.

We begin with the following lemma, which is just a straightforward adaptation of \cite[Lemma 4.10]{pelletier}. The proof is literally the same and uses the Leibniz rule for the Poisson bracket.
\begin{lem}\label{lem:pblin}Let $\pb$ be a linear localizable non-queer Poisson bracket on the bundle $E_*\to M$. For $f,g\in C^\infty_{loc}(M)$ one has $\{f\circ\pi_*, g\circ\pi_*\} = 0$. Moreover let $X\in\Gamma_{loc}(E)$ be a local section of $E$. Then $\{f\circ\pi_*,\lambda_X\}$ is fiber-wise constant, i.e. it is equal to $g\circ\pi_*$ for some $g\in C^\infty_{loc}(M)$.
\end{lem}

We define the Lie algebroid structure on $E$ by inverting the procedure described in the previous section. First, we observe the following
\begin{lem}\label{lem:linPB}
The mapping $\lambda:\Gamma_{loc}(E)\ni X \mapsto \lambda_X\in\mathcal L_{loc}(E_*)$ is a bijection.
\end{lem}
\begin{proof}Injectivity of $\lambda$ is straighforward.

To prove surjectivity, notice that any smooth locally defined fiber-wise linear map on $E_*$ is at each point $m\in M$, by definition, given by a unique element $s(m)\in \big(\pi^{-1}_*(m)\big)^* = \pi^{-1}(m)\subset E$. In this way we obtain a smooth section $s\in \Gamma_{loc}(E)$.
\end{proof}

Note that in \cite{pelletier} the bijectivity of the map $\lambda$ (denoted there by $\Phi$) did not hold and only injectivity was proven (see \cite[Lemma 4.6]{pelletier}). However, it was still possible to prove that linear sub-Poisson brackets defined on the characteristic subbundle $T^\flat M$ preserve the family of functions of the form $\lambda_X$.

\begin{thm}
Let the bundle $E_*$ be a Banach Poisson manifold with a linear localizable non-queer Poisson bracket. Then there exists a canonically associated structure of Banach Lie algebroid on $E$ with Lie bracket and anchor given as follows:
\[ [X_1, X_2] = \lambda^{-1}\big(\{\lambda_{X_1},\lambda_{X_2}\}\big)\]
\[a(X) = T\pi_*\sharp(d\lambda_X)\]
for $X_1, X_2, X\in \Gamma_{loc}(E)$.
\end{thm}
\begin{proof}
This proof is an adaptation of the proof of Theorem~7.1 in \cite{cabau-pelletier_book}.

First of all, the formula for the bracket makes sense due to Lemma~\ref{lem:linPB}. Namely, the bracket of two linear functions $\lambda_{X_1}$ and $\lambda_{X_2}$ is again a linear function. Thus one can apply to it the inverse of the map $\lambda$.

The formula for the anchor makes sense since we assume in Definition~\ref{def:pm} that the map $\sharp$ takes values in $TE_*$. Namely, $d\lambda_X$ is a section of $T^*E_*$. Due to condition \eqref{sharp}, $\sharp(d\lambda_X)$ is now a section of $TE_*$, which is mapped by $T\pi_*$ to a section of $TM$. That section is the Hamiltonian vector field for the function $\lambda_X$ pushed down to the base $M$.

We need to verify that the value of the vector field $a(X)$ at some point $m\in M$ depends only on the value of the section $X$ at the point $m$. To this end, note that the action of the vector field $a(X)$ on a smooth function $f\in C^\infty(M)$ can be written as:
\[ a(X)f \circ \pi_* 
=\pair{d(f\circ\pi_*)}{\sharp(d\lambda_X)} \]
Taking a look at this formula in trivialization, at first glance, using \eqref{dlambda}, $d\lambda_X$ depends both on the value of $X$ at a given point as well as its derivative. 
\[ a(X)f\circ\pi_*(m,\phi) = \pair{(m,\phi,df(m),0)}{\sharp (m,\phi,d(\phi\circ v)(m), v_m)}.\]
However, since $\pi_*(m,\phi)=m$, the left-hand side does not depend on $\phi$. In consequence, the right-hand side cannot depend on the term $d(\phi\circ v)(m)$, which implies that $a(X)$ depends only on the value of the section $X$ at point $m$. It can be seen more directly by observing that the part depending on the derivative can be separated as follows:
\[ \pair{(m,\phi,df(m),0)}{\sharp (m,\phi,d(\phi\circ v)(m), 0)} = \{f\circ \pi_*,l_v^\phi\circ \pi_*\}, \]
where $l_v^\phi(m) = \phi\circ v_m$. Now, using Lemma~\ref{lem:pblin}, we see that this term vanishes as it is a Poisson bracket of two fiber-wise constant functions.

Let us check the conditions of Definition~\ref{def:alg}. The first one follows from direct computation:
\[ [X_1, f X_2] = \lambda^{-1}\big(\{\lambda_{X_1},\lambda_{f X_2}\}\big) = \lambda^{-1}\big(\{\lambda_{X_1},f\lambda_{X_2}\}\big)= \]
\[=\lambda^{-1}\big(f\{\lambda_{X_1},\lambda_{X_2}\} + \lambda_{X_2}\{\lambda_{X_1},f\}\big) = f[X_1,X_2]+\sharp(d\lambda_{X_1})f \cdot X_2=\]
\[=f[X_1,X_2]+a(X_1)f \cdot X_2.\]
The second one is a consequence of an analogous property for the Poisson bracket. Since the Poisson bracket was assumed to be non-queer, the Lie algebroid bracket depends only on the first jet of sections, thus we get the third property.
\end{proof}

\begin{rem}
The condition \eqref{sharp} from Definition~\ref{def:pm} is generally too strong. The example in the next section shows that in some cases the manifolds with the Poisson bracket that do not satisfy condition \eqref{sharp}, still have an algebroid structure on their dual.
\end{rem}

\section{Example: a precotangent bundle}\label{sec:weak_symplectic}

We will now consider the case $E=TM$ and assume that the precotangent bundle $T_*M$ exists. Note that in general, it might not exist for an arbitrary Banach manifold $M$. Naturally, even not every Banach space possesses a predual space: for example the space of compact operators on a separable Hilbert $\Hh$ space does not. However, even if modeling Banach space admits a predual, one still needs to assume that there exists an atlas on $M$ such that tangent maps of charts preserve the predual spaces. On top of that, a predual space might not be unique in general. The existence of the precotangent bundles is established for Banach Lie groups provided its Banach Lie algebra admits a predual, see \cite[Section 7]{OR}, and Grassmannians of reflexive Banach spaces or $p$-restricted Grassmannians ($p\geq 1$), see \cite{G-precot}.

In some cases, working with the precotangent bundle might be more convenient. For example, if one considers the Banach space of bounded operators on $\Hh$ (or a manifold modeled on it), the cotangent space involves the dual of the space of bounded operators, which is not easily described. On the other hand, the precotangent bundle involves a predual space, which is just the space of trace-class operators.

We begin by the usual approach: define the canonical $1$-form $\theta\in \Gamma (T^*T_*M)$ on the precotangent bundle $T_*M$ in the same way it is done for the cotangent bundle:
\be \pair{\theta_\rho}{X_\rho}:=\pair{\rho}{T_\rho\pi_*X_\rho}\ee
for $\rho\in T_*M$ and $X\in \Gamma (TT_*M)$. It can be considered as the restriction of the standard canonical 1-form on $T^*M$ to the subbundle $T_*M\subset T^*M$.

In general, the form $\omega=d\theta$ is only a weak symplectic form. Similarly to the situation on the cotangent bundle (see \cite[\S 1, Theorem 3]{marsden-chernoff}), it is strong if and only if $\mathbb M$ is a reflexive Banach space. It means that it is a closed 2-form and the map
\be \flat: TT_*M \ni X \mapsto \omega (X, \cdot) \in T^*T_*M \ee
is injective, but in general not surjective. This can be easily seen by expressing $\omega$ and $\flat$ in trivialization as
\be \label{omega-triv}\omega(m,\phi)\big((m,\phi,v,\Phi),(m,\phi,w,\Psi)\big) = \pair{\Phi}w -\pair\Psi v\ee
\[ \flat : U\times \M_*\times \M \times \M_* \to U\times \M_*\times \M^* \times \M\]
\be \label{flat-triv}\flat(m,\phi, v,\Phi) = (m,\phi,\Phi, -v),\ee
where $(m,\phi)\in U\times \M_*$, $v,w\in \M$, $\Phi,\Psi\in\M_*$.
Since $\Phi$ is an element of $\M_*$, which in general is only a proper subset of $\M^*$, the image of $\flat$ is not the whole $T^*T_*M$.

Let us denote the image of $\flat$ by $T^\flat T_*M\subset T^*T_*M$. Since $\flat$ is injective, it is possible to consider its inverse defined on $T^\flat T_*M$
\be \sharp := \flat^{-1}:T^\flat T_*M\to TT_*M,\ee
\be \label{sharp-triv}\sharp(m,\phi, \Phi, v) = (m,\phi,-v,\Phi).\ee

The set of functions such that their derivative belongs to $T^\flat T_*M$ will be denoted by
  \be C^\infty_\flat(T_*M):=\{ f\in C^\infty(T_*M)\;|\; df\in T^\flat T_*M \}.\ee
Using the formula \eqref{flat-triv} for $\flat$ in trivialization, it can be seen that $C^\infty_\flat(T_*M)$ can be equivalently defined as 
\be C^\infty_\flat(T_*M):=\{ f\in C^\infty(T_*M)\;|\; \tfrac{\partial f}{\partial m}\in \M_* \}.\ee

Now in a standard way one can use $\sharp$ to introduce a sub-Poisson bracket of functions $f,g\in C^\infty_\flat(T_*M)$ related to the canonical weak symplectic form $\omega$:
  \be \{f,g\}_\omega := \omega(\sharp df, \sharp dg).\ee
It can be seen that $\{f,g\}_\omega$ is again an element of $C^\infty_\flat(T_*M)$, see e.g. \cite[Proposition 3.11]{tumpach-bruhat} or \cite[Lemma 7.1]{cabau-pelletier_book}. Note that this bracket is localizable, so it can be defined for local versions of $C^\infty_\flat(T_*M)$.

The expression for the sub-Poisson bracket $\pb_\omega$ in the trivialization follows from 
\eqref{omega-triv} and \eqref{sharp-triv}:
\be \label{pb_omega}
\{f,g\}_\omega (m,\phi) = -\pair{\frac{\partial f}{\partial m}}{\frac{\partial g}{\partial \phi}}+\pair{\frac{\partial g}{\partial m}}{\frac{\partial f}{\partial \phi}}.\ee
According to the definition of the bracket $\pb_\omega$, this formula is only valid for $f,g$ such that $\frac{\partial f}{\partial m}, \frac{\partial g}{\partial m}\in \M_*$. However, the expression on the right-hand side also makes sense in the general case where $\frac{\partial f}{\partial m}, \frac{\partial g}{\partial m}\in \M^*$. In that case, the formula \eqref{pb_omega} coincides with the formula \eqref{alg-bracket-gen} for the Poisson bracket given by the canonical Banach Lie algebroid structure of $TM$.

\begin{thm}
The Poisson bracket \eqref{alg-bracket} induced by the algebroid structure on the precotangent bundle $T_*M$ is given by the formula
\be \label{pb_precot}
\{f,g\}(m,\phi) = -\pair{\frac{\partial f}{\partial m}}{\frac{\partial g}{\partial \phi}}+\pair{\frac{\partial g}{\partial m}}{\frac{\partial f}{\partial \phi}}\ee
In effect, it coincides with $\pb_\omega$ 
for functions belonging to $C^\infty_\flat(T_*M)$.
\end{thm}
\begin{proof}
We consider the trivialization of $T_*M$ given by the dual of the tangent map of a chart for a chart domain $U$, restricted to the predual space. Since the anchor $a$ is the identity, the formula \eqref{alg-bracket-gen} simplifies to \eqref{pb_precot}. It is a well-defined Poisson bracket by Theorem~\ref{algebroid-poisson}. It coincides with \eqref{pb_omega} for $f,g\in C^\infty_\flat(T_*M)$. 
\end{proof}

In this way, using the algebroid approach, we obtain the Poisson bracket on $C^\infty(T_*M)$, which is the extension of the bracket $\pb_\omega$ given by the canonical symplectic form on $T_*M$.

\begin{rem}
Note that $T_*M$ is not a Poisson manifold in the sense of Definition~\ref{def:pm} in a non-reflexive case, since the conditions of Theorem~\ref{thm:pb} are not satisfied. Namely, the anchor is the identity map $a=\id:TM\to TM$. The dual map $a^*$ is the identity map $T^*M\to T^*M$. In consequence, it does not take values in the precotangent bundle $T_*M$. The second condition from Theorem~\ref{thm:pb}, on the other hand, is always satisfied as the Lie bracket is the commutator of the vector fields which in coordinates does not contain the terms without derivatives, i.e. in formula \eqref{lie_bracket_local} only two first terms are present. In consequence $C_m=0$. An analogous situation was described in \cite[Section 7]{OR} for Banach Lie groups. Even though, we still obtain the Poisson bracket in the sense of Definition~\ref{def:pb}.
\end{rem}

\section{Example: trivial bundle \texorpdfstring{$\ell^2 \times \ell^\infty\to \ell^2$}{l\^2 x l\^inf -> l\^2}}\label{sec:trivial}

We now consider a trivial Banach vector bundle $E\to M=\ell^2 \times \ell^\infty\to\ell^2$. Sections of $E$ can be identified with maps $X:\ell^2\to\ell^\infty$. We fix a continuous linear map $A:\ell^\infty\to \ell^2$. Now, we define the bracket on such sections as:
\[ [X, Y](x) = Y'(x)AX(x) - X'(x)AY(x). \]
We treat the derivative $X'(x)$ as a linear map $\ell^2\to\ell^\infty$. The anchor in that case is
\[ a(x,X) = (x,AX)\in \ell^2\times\ell^2\cong T\ell^2. \]

\begin{prop}
The bundle $\ell^2 \times \ell^\infty$ with given anchor and Lie bracket is a Banach Lie algebroid.
\end{prop}
\begin{proof}
Both the Lie bracket and the anchor are continuous. The bracket is localizable and depends only on the first jet of sections (condition 3. of Definition~\ref{def:alg}).

Condition 1. of Definition~\ref{def:alg} follows by straightforward computation. Let $X,Y:\ell^2\to\ell^\infty$, $f:\ell^2\to\R$. We have
\[ [X,fY](x) = (fY)'(x) AX(x) - X'(x)Af(x)Y(x) = \]
\[= f'(x)(AX(x))Y(x) + f(x)Y'(x)AX(x) - f(x)X'(x)AY(x) =\]
\[ = AX(f)(x) + f(x)[X,Y](x).\]
In the last expression $AX(f)(x)$ should be understood as a vector field $AX$ acting on the function $f$.

Condition 2. of Definition~\ref{def:alg} also follows by straightforward calculation.
\end{proof}

To show an example of a Lie algebroid of this kind, let us write down the formula for the bracket in coordinates for some particular choice of $A$. By superscript $j$ we will denote $j^{th}$ element of the sequence. Consider for example a diagonal operator $A$ given by $(AX)^k = \tfrac{1}{k} X^k$. Then we get
\[ [X, Y]^j(x) = \sum_{k=1}^\infty \frac 1 {k} \big(X^k(x) \partial_k  Y^j(x) - Y^k(x) \partial_k  X^j(x)\big), \]
where by $\partial_k$ we denote the partial derivative $\frac\partial{\partial x^k}$.

The predual bundle to $\ell^2\times\ell^\infty$ is $\ell^2\times\ell^1$. From Theorem~\ref{thm:pb} we get the condition for it to be a Banach Poisson manifold.

\begin{prop}
The predual bundle $\ell^2\times\ell^1$ equipped with the canonical Poisson structure \eqref{alg-bracket-gen} is a Banach Poisson manifold in the sense of Definition~\ref{def:pm} if and only if $A^*(\ell^2)\subset \ell^1$.
\end{prop}
\begin{proof}
The condition $a^*(T^*M)\subset E_*$ of Theorem~\ref{thm:pb} in this case reduces to $A^*(\ell^2)\subset \ell^1$.

The other condition $\ad_{m,x}^*(\E_*)\subset \E_*$ is trivial as the bilinear map $C_x$ defined in \eqref{lie_bracket_local} vanishes. In consequence $\ad_{m,x}^*=0$.
\end{proof}

Note that for a particular example $(AX)^k = \tfrac{1}{k} X^k$, the condition of the previous Proposition holds. In consequence, we obtain a linear Banach Poisson structure on $\ell^2\times\ell^\infty$.

It can be easily seen that this example can be generalized to the arbitrary trivial bundle $\M\times\E\to\M$, where $\M$ and $\E$ are Banach spaces such that $\E$ admits a predual.

\section*{Acknowledgement}
The authors thank Prof. Fernand Pelletier for valuable advice and help while writing this manuscript. The authors also thank the anonymous referee for valuable comments and corrections.

This research was partially supported by joint National Science Centre, Poland (number 2020/01/Y/ST1/00123) and Fonds zur Förderung der wissenschaftlichen Forschung, Austria (number I 5015-N) grant ``Banach Poisson--Lie groups and integrable systems.''

The first author thanks the Erwin Schr\"odinger Institute for its hospitality during the thematic programme ``Infinite-dimensional Geometry: Theory and Applications'' in January--February 2025. 

The stimulating atmosphere of the Workshop on Geometric Methods in Physics in Białowieża/Białystok, Poland, is also acknowledged, where both authors could work and discuss this paper.

\small

\end{document}